\documentclass[12pt]{article}

\usepackage{graphicx}
\graphicspath{ {./figures/} }

\usepackage{amsmath,amssymb,amsthm,enumerate}
\usepackage[svgnames]{xcolor}
\usepackage{hyperref}
\usepackage{pifont}
\usepackage{mathtools}
\usepackage{authblk}
\usepackage[misc]{ifsym}

\usepackage{titlesec}

\numberwithin{equation}{section}
\newtheorem{theorem}{Theorem}[section]
\newtheorem{corollary}[theorem]{Corollary}
\newtheorem{lemma}[theorem]{Lemma}
\newtheorem{prop}[theorem]{Proposition}
\newtheorem{conjecture}[theorem]{Conjecture}

\newtheorem{remark}[theorem]{Remark}
\newtheorem{defn}[theorem]{Definition}

\newcommand\ga{\gamma}

\newcommand\la{\lambda}

\newcommand{\R}{\mathbb{R}}

\newcommand{\Z}{\mathbb{Z}}

\newcommand{\pa}{\partial}
\newcommand{\na}{\nabla}

\renewcommand{\bar}[1]{\overline{#1}}

\newcommand{\cE}{\mathcal{E}}

\DeclareMathOperator\erf{erf}

\def\Z{\mathbb{Z}}
\def\R{\mathbb{R}}

\begin{document}

\title{Equivalent formulations of the oxygen depletion problem, other implicit free boundary value problems, and implications for numerical approximation}

\date{}

\author{Xinyu Cheng\footnote{School of mathematical sciences, Fudan University \Letter\  xycheng@fudan.edu.cn} , Zhaohui Fu\footnote{University of British Columbia, \Letter\ 	fuzh@math.ubc.ca}, Brian Wetton\footnote{University of British Columbia, \Letter\	wetton@math.ubc.ca.}}


\maketitle

\begin{abstract}                
 The Oxygen Depletion problem is an implicit free boundary value problem. The dynamics allow topological changes in the free boundary. We show several mathematical formulations of this model from the literature and give a new formulation based on a gradient flow with constraint. All formulations are shown to be equivalent. We explore the possibilities for the numerical approximation of the problem that arise from the different formulations. We show a convergence result for an approximation based on the gradient flow with constraint formulation that applies to the general dynamics including topological changes. More general (vector, higher order) implicit free boundary value problems are discussed. Several open problems are described.  
\end{abstract}

\section{Introduction}

The Oxygen Depletion (OD) problem is a free boundary value problem of implicit type. Implicit here means that the free boundary is specified implicitly by an extra boundary condition rather than explicitly as an interface normal velocity as for a Stefan problem \cite{Ste90,Stefan1,Stefan2}. 

The OD problem was introduced as a model of oxygen consumption and diffusion in living tissue but several other problems have similar structure. Some of the early work is described in \cite{Crank} with a great deal of subsequent interest from the analysis and numerical research communities in \cite{Ste90,schaffer,FP81,rogers,oxygen}. Reference \cite{oxygen} has a review of much of the previous work. In the current work, we pursue an understanding of the analysis of the OD problem as the simplest example of an implicit moving free boundary value problem. We are motivated by an interest in the analysis and computation of a general class of implicit free boundary value problems. 

By way of introduction, we present the OD problem in 1D for an unknown $u(x,t)$ for $x \in [0, s(t)]$ with a single free boundary $x=s(t)$ and a no flux condition $u_x =0$ at $x=0$. At the free boundary, $u=0$ and additionally $u_x = 0$. These two conditions implicitly define the free boundary $x=s(t)$. 
The solution obeys 
\begin{equation}
\label{OD_introduction} 
u_t = u_{xx} - 1
\end{equation} 
for $x \in [0, s(t)]$ and it is natural to extend $u \equiv 0$ for $x > s(t)$ in a $C^1$ continuous way. We consider positive initial conditions for $u$ in $[0, s(0))$. This is one of the forms of the OD problem foreshadowed by the title. We consider several formulations in the literature and in new results we show they are all equivalent. We introduce a new formulation as the $L_2$ gradient flow with constraint on the energy from the elliptic obstacle problem. The obstacle problem has had considerable interest in the literature \cite{Monotone_Forcardi,cafRE,Weiss,Lin_Mon,Mon03}. Some discussion of the numerical methods that follow from the different formulations is given. 

\begin{remark}
We highlight that the extra boundary condition for implicit free boundary value problems does not explicitly contain the interface velocity, hence this velocity is determined {\em implicitly}. For steady state free boundary problems, where the interface velocity is zero, the difference between implicit and explicit formulations disappears.
\end{remark} 

\begin{remark}
The Oxygen Depletion problem is also known as the Oxygen Diffusion problem in the literature. We prefer the former name as it is the depletion, the second term on the RHS of (\ref{OD_introduction}), rather than the diffusion, the first term on the RHS of (\ref{OD_introduction}), that leads to the formation of free boundaries. 
\end{remark} 


We invite the reader to view computational examples of the dynamics in Figures~\ref{f:ODex1}, \ref{f:ODex2}, and \ref{f:ODex3}. Solutions of (\ref{OD_introduction}) can go negative, but physically relevant values of concentration have $u \geq 0$. In the 1D case, preserving nonnegativity results in the break up or merger of intervals where $u>0$ as shown in Figure~\ref{f:ODex2}. Topological change can be more complex in higher dimensions as seen in Figure~\ref{f:ODex3}. Some of the problem formulations we consider can handle topological changes while others cannot. 

The paper is organized as follows. In Section~\ref{s:equivalence} we present the different formulations and show their equivalence. In Section~\ref{s:numerical} we present two numerical schemes. One scheme gives high accuracy solutions to the 1D problem without topological change. The other scheme, based on our new gradient formulation of the problem, can be applied in higher dimensions and can handle topological changes. A convergence proof for this new scheme is given. In Section~\ref{s:other} we present some other implicit free boundary value problems of interest and indicate how our results can be extended to them, with some open questions. The analysis of a biharmonic problem with gradient flow structure follows directly from our new formulation of the OD problem.    
We end with a short Summary that includes a list of open problems. 

\section*{Notation}
We define the space 
$H^1_+(\Omega)\coloneqq \{u\in H^1(\Omega): u\geq 0 \ a.e. , \frac{\partial u}{\partial n}|_{\partial \Omega }=0\}$. 
For simplicity we consider $d-$dimensional open connected bounded domain $\Omega=(0,1)^d$ with homogeneous Neumann boundary conditions, where $d=1,2,3.$ We further denote $\mathcal{J}$ to be collection of functions $v\in L^2(0,T;H^1(\Omega))$ such that $v(t)\in H^1_+(\Omega)$ for a.e. $t\in(0,T)$. In some instances, we denote the time derivative by $\dot{u}(t)$ and the space derivative in 1D case by $u'(x)$. s
Given two quantities $A$ and $B$, we use $A\lesssim B$ to denote that there exists a constant $C>0$ such that $A \leq C\cdot B$.

\section{Equivalent Formulations}
\label{s:equivalence} 

\subsection{1D formulations without topological change} 

%

\subsubsection{Standard formulation in 1D}
\label{section_od}
The one-dimensional oxygen depletion problem with associated free boundary and initial conditions is as follows:
\begin{equation}
\label{eq:standard} 
\left\{
\begin{aligned}
&u_t=u_{xx}-1,&&\ 0\leq x\leq s(t)  \\
&u(x,t)=0, &&\ x>s(t)\\
&u_x(0,t)=0,&& \ t>0\\
&u(s(t),t)=u_x(s(t),t)=0,&&\ t>0\\
&u(x,0)=u_0(x),&& \ 0\leq x\leq 1 \\
&s(0)=1.&&
\end{aligned}\right.\end{equation}
We assume here that $u_0$ satisfies all necessary smoothness and compatibility assumptions needed in the analyses cited below. By literature convention, we consider here a problem with a fixed, no-flux boundary condition at $x=0$ and only one free boundary $s(t) >0$. Uniqueness and lack of topological change when $u_0^\prime \leq 0$ follows from a modified maximum principle argument \cite{FP81}. 

Existence can be seen by considering $v = u_t$ which satisfies a standard Stefan problem \cite{Crank} with explicit interface velocity:
\begin{equation*}
\left\{
\begin{aligned}
&v_t=v_{xx},&&\ 0\leq x\leq s(t)  \\
&v_x(0,t)=0,&& \ t>0\\
&v(s(t),t)=0,\ v_x(s(t),t)=-\dot{s}(t),&&\ t>0 \\
&s(0)=1.&&
\end{aligned}\right.\end{equation*}
One can check the function $u=\int_0^t v\ d\tau$ solves the oxygen depletion problem. To prove existence and uniqueness of Stefan problem, one can verify that the map 
$$\mathcal{T}(s)(t)\coloneqq1-\int_0^t v_x(s(\tau),\tau))\ d\tau ,\ T\geq t\geq 0,$$
defines a contraction map \cite{Magenes}. 

\begin{remark}
\label{rem:high_order} 
The reformulation in $v = u_t$ to an explicit free boundary problem with interface velocity equal to $- v_x$ can be reinterpreted as a normal velocity for the problem for $u$ with velocity equal  to $-v_x = - u_{tx} = -u_{xxx}$. The authors are not aware of any analysis or computational methods based on this velocity expression with higher order spatial derivatives. 
\end{remark}

 \subsubsection{Mapped domain formulation in 1D}
\label{s:yproblem}

Considering the same smooth solutions without topological change in 1D discussed above, we consider 
$s(t) > 0$ in $t\in[0,T]$, take $y=x/s(t)$, and reformulate the oxygen depletion problem as 
\begin{equation}
\label{eq_yproblem}
u_{yy}+\dot{s}syu_y -s^2u_t-s^2=0
\end{equation}
with boundary conditions $u_y(0,t) = u(1,t) = u_y(1,t) = 0$. 
Over a short time period, we assume that $\dot{s}(t)$ and $s(t)$ are uniformly bounded, thus the linear operator is parabolic. Assuming $s(t)$ is known, uniqueness of $u$ is not an issue; however, to prove uniqueness of the solution pair $(\tilde{u},s)$, we introduce the map $\mathcal{G}$: $X\to Y$, where $X$ is the closed subspace of $H^1(H^2([0,s(t)]); [0,T])\times C^1([0,T])$ that solves OD system and $Y$ is the closed subspace of $H^1(H^2([0,1]); [0,T])\times C^1([0,T])$ that solves the reformulated system:
\begin{align*}
    \mathcal{G}((u(x,t),s(t))=(u(y,t),s(t))\coloneqq (U(y\cdot s(t),t),s(t)).
\end{align*}
where $U(x,t)$ is the solution from the previous section. 
One can check that the map $\mathcal{G}$ is a bijection and so all solutions of (\ref{eq_yproblem}) are equivalent to the solutions in the standard formulation of Section~\ref{section_od}. 

A numerical method based on this formulation is presented in Section~\ref{s:ymethod}. The computations in Figure~\ref{f:ODex1} are done with a method based on this formulation.

\begin{remark}
\label{rem:yproblem} 
A direct analysis of this formulation would be useful as a stepping stone to a convergence proof for the numerical approximation in Section~\ref{s:ymethod} and an analysis of the general class of problems in Section~\ref{s:other}. We have not been able to make progress on such an analysis. There are subtleties in the problem: note that changing $-s^2$ to $+s^2$ makes the problem ill defined as $s(t) = +\infty$ for $t>0$ in that case. 
\end{remark} 

\subsection{Higher dimensional formulations that allow topological change}

A weak form of the solution can be introduced using a variational inequality approach (\ref{eq_paravari})\cite{vari_pLap,Rudd_Sch_vari}. This is described in Section~\ref{s:variational} below. We use this formulation as the basis for equivalence to the others. This formulation is amenable to approximation using the Augmented Lagragian Method \cite{AugLag_actset,AugLag_paravar}. 
We then introduce a new formulation as $L^2$ gradient flow with constraint on the energy from the elliptic obstacle problem in Section~\ref{s:gradient}. 
The computations in Figures~\ref{f:ODex2} and~\ref{f:ODex3} are done with a method based on this formulation. We show a regularized approach with parameter $\epsilon$, similar to the approach in \cite{rogers}, in Section~\ref{s:regular}.  

\subsubsection{A parabolic variational inequality formulation}
\label{s:variational} 

To proceed with the discussion of the problem in higher dimensions with topological changes, we consider the standard approach to weak solutions in this setting: a variational inequality formulation \cite{Lions,vari_pLap}. This approach has been well studied and we describe results in the literature.

We consider the following problem: find a function $u\in \mathcal{J}$ with $u(0)=u_0\in H_+^1(\Omega)$ that solves 
\begin{align}\label{eq_paravari}
    \int_0^t\int_\Omega u_t\cdot (v-u)+\int_0^t\int_\Omega \nabla u\cdot \nabla (v-u)\geq \int_0^t\int_\Omega u-v;\ \mbox{for all}\ v\in\mathcal{J},\ \mbox{a.e.}\ t\in(0,T). 
\end{align}
\begin{prop}
The variational inequality \eqref{eq_paravari} has at most one solution and in fact suppose $u_1$ and $u_2$ solves \eqref{eq_paravari} with distinct initial conditions $u_{1_0}$ and $u_{2_0}$ then 
\begin{equation}
    \Vert u_1-u_2\Vert_{L^\infty(0,T;L^2(\Omega))}\leq \Vert u_{1_0}-u_{2_0}\Vert_{L^2(\Omega)}.
\end{equation}
\end{prop}

\begin{proof}
Note $u_j\in\mathcal{J}$ satisfies \eqref{eq_paravari} for j=1,2, in particular
\begin{align*}
    && \int_0^t\int_\Omega \partial_t u_1\cdot (u_2-u_1)+\int_0^t\int_\Omega \nabla u_1\cdot \nabla (u_2-u_1)\geq \int_0^t\int_\Omega u_1-u_2,\\
    && \int_0^t\int_\Omega \partial_t u_2\cdot (u_1-u_2)+\int_0^t\int_\Omega \nabla u_2\cdot \nabla (u_1-u_2)\geq \int_0^t\int_\Omega u_2-u_1.
\end{align*}
Summing the two inequalities above and denote $w=u_1-u_2$, one has 
\begin{align*}
 &\int_0^t\int_\Omega \partial_t w\cdot w+\int_0^t\int_\Omega \nabla w\cdot \nabla w\leq0\\
 \implies &\int_0^t\int_\Omega (w^2)_t\leq 0\implies  \Vert w\Vert_{L^\infty(0,T;L^2(\Omega))}\leq \Vert w_0\Vert_{L^2(\Omega)}.
\end{align*}
\end{proof}
\begin{theorem}
There exists a unique solution to the variational inequality \eqref{eq_paravari}.
\end{theorem}
\noindent Note that this can be done by a standard monotone operator argument and we refer to \cite{vari_pLap}.

We show equivalance to the 1D formulations. 
Any smooth solution $u$ to \eqref{eq:standard} must solve \eqref{eq_paravari} in the 1D case and by uniqueness the solution to \eqref{eq_paravari} therefore solves OD. To see this, we first observe that $u\geq0$ and therefore $u\in\mathcal{J}$. Resulting from that, for any $v\in\mathcal{J}$ and for a.e. $t\in(0,T)$ by applying integration by parts we obtain that
\begin{align*}
    \int_0^t\int_0^1 u_t\cdot (v-u)+\int_0^t\int_0^1  u_x\cdot (v_x-u_x)=&\int_0^t\int_0^{s(\tau)}(u_t-u_{xx})(v-u)\ dxd\tau\\=&\int_0^t\int_0^{s(\tau)} u-v \ dxd\tau\geq \int_0^t\int_\Omega u-v .
\end{align*}


\subsubsection{A gradient flow formulation}
\label{s:gradient} 

In this section, we formulate the OD problem as the $L^2$ gradient of the energy from the elliptic obstacle problem. A formal calculation with 
\[
\cE (t) \coloneqq \int_\Omega\frac12|\nabla u|^2+u
\]
leads to 
\[
\frac{d \cE}{dt} = - \int_\Omega (\Delta u -1)^2.
\]
It is convenient to present the equivalence of the gradient flow formulation as the limit of implicit time steps as this gets us half way to the convergence result for the fully discrete method described in Section~\ref{s:brians}. The spatially continuous, time discrete solutions $u_n$ approximate $u(\cdot, nk)$, where $k$ is a time step. We consider the following minimization problem for $u = u_{n+1}$ to the following energy functional:
\begin{equation} 
\label{eq_elli_mini}
    E[u]=\int_{\Omega} \frac12 |\nabla u|^2+\frac{1}{2k} (u-u_n)^2+u,
\end{equation}
where $u \in H^1_+(\Omega)$.
Existence and uniqueness of the minimizer is guaranteed by the standard calculus of variation technique and convexity of the energy functional \cite{Struwe}.

\begin{remark}By defining the discrete energy
$\cE^{n}\coloneqq \int_\Omega\frac12|\nabla u_{n}|^2+u_{n},$ we can see that $\cE^{n+1}\leq \cE^n.$ This can be derived by considering 
$E[u]=\int_\Omega \frac12 |\nabla u|^2+u+\frac{(u-u_n)^2}{2k}$ with $E[u_{n+1}]\leq E[u_n]$. This gives the discrete gradient flow structure.
\end{remark}


We will derive the corresponding Euler-Lagrange equation for the minimizing problem following the idea from \cite{Monotone_Forcardi}. We give an adapted proof in our case for completeness.

\begin{theorem}\label{thm_mini}
Suppose $u_{n+1}$ is the unique minimizer to the energy minimizing problem \eqref{eq_elli_mini}, then $u_{n+1}$ is the (weak) solution to the following modified backward Euler scheme:  
$$\frac{u_{n+1}-u_n\cdot \chi_{\{u_{n+1}>0\}}}{k}=\Delta u_{n+1}-\chi_{\{u_{n+1}>0\}} .$$
\end{theorem}

\noindent To begin with, we consider an equivalent energy minimizing problem:
\begin{equation}\tilde{E}[u]\coloneqq\int_\Omega \frac12 |\nabla u|^2+\frac{1}{2k}u^2+(1-\frac{u_n}{k})u^++\frac{1}{2k}u_n^2,\label{eq_elli_mini2}
\end{equation}
subject to 
$$u\in\tilde{\mathcal{K}}\coloneqq \{v\in H^1(\Omega): \frac{\partial v}{\partial n}|_{\partial\Omega}=0\}$$
 where $u^+=\max(u,0)$.
 
\begin{lemma}\label{lem_mini2}
There exists a unique $\tilde{u}\in \tilde{\mathcal{K}}$ such that $$\tilde{E}[\tilde{u}]=\min_{v\in \tilde{\mathcal{K}}}\tilde{E}[v];$$
moreover such $\tilde{u}$ is the unique minimizer to \eqref{eq_elli_mini}.
\end{lemma}
\begin{proof}
Firstly, by similar argument, the existence and uniqueness of this energy minimizing problem can be proved.

Now to show the equivalence of these two minimizing problems, we recall that the minimizer $u\geq 0$, so 
$$\min_{v\in H^1_+(\Omega)} E[v]=E[u]=\tilde{E}[u]\geq \min_{v\in \tilde{\mathcal{K}}}\tilde{E}[v];$$

\noindent On the other hand, in order to show 
$$\min_{v\in \mathcal{K}} E[v]\leq \min_{v\in \tilde{\mathcal{K}}}\tilde{E}[v],$$
we note that for any $v\in \tilde{\mathcal{K}}$, the corresponding $v^+\in H^1_+(\Omega)$. As a result,
$$ E[u]\leq E[v^+]\leq\tilde{E}[v]$$
for any $v\in\tilde{\mathcal{K}} $, therefore we get
$$E[u]\leq \min_v \tilde{E}[v].$$
Now since $E[u]=\tilde{E}[u]=\min \tilde{E}[v]$, we have $\tilde{u}=u$ by the uniqueness.

\end{proof}

It remains to derive the Euler-Lagrange equation for this new energy minimizing scheme. 

\begin{prop}
\label{th:minimizer} 
Suppose $u$ is the unique solution to the minimization problem \eqref{eq_elli_mini2}, then $u$ is the (weak) solution to the following modified backward Euler scheme:  
$$\frac{u-u_n\cdot \chi_{\{u>0\}}}{k}=\Delta u-\chi_{\{u>0\}} .$$
\end{prop}
\noindent The proof of this result is found in Appendix~\ref{s:prop}. 


We follow the idea in \cite{SmoothVariational_Brezis} to formulate the minimization problem \eqref{eq_elli_mini} as a variational inequality:
\begin{equation}\label{eq_ellivari}
u\in H^1_+(\Omega): \int_{\Omega} \nabla u\cdot \nabla(v-u)+\frac{u}{k}(v-u)\ dx\geq \int_{\Omega} \left(\frac{u_n}{k}-1\right)\cdot (v-u)\mbox{for all } v\in H^1_+(\Omega) .
\end{equation}{}
To see the equivalence of the energy minimization and elliptic variational inequality we now state the proposition.

\begin{prop}\label{prop_equiv_vari}
Any solution to the minimization problem \eqref{eq_elli_mini} is also a solution to the variational inequality \eqref{eq_ellivari} and vice versa.
\end{prop}
\begin{proof}
Suppose $u$ is an energy minimizer to \eqref{eq_elli_mini}. Let $v\in H^1_+(\Omega)$, note that $H^1_+(\Omega)$ is convex then $(1-\lambda) u+\lambda v\in H^1_+(\Omega)$ for any $\lambda\in[0,1]$. Using $(1-\lambda) u+\lambda v$ as a competitor in $E[u]\leq E[(1-\lambda) u+\lambda v]$, we can derive from the order $O(\lambda)$:
\begin{align*}
   \int_{\Omega} \nabla u\cdot \nabla(v-u)+\frac{u}{k}(v-u)\ dx\geq \int_{\Omega} \left(\frac{u_n}{k}-1\right)\cdot (v-u), \ \forall\ v\in H^1_+(\Omega).
\end{align*}
The reverse can be proved similarly.

\end{proof}

Note that this formulation uses convexity of $H^1_+(\Omega)$; the optimal regularity of $u$ is $C^{1,1}_{loc}$:

\begin{theorem}[regularity]
Suppose u is a solution to \eqref{eq_elli_mini} (or \eqref{eq_ellivari}), then there exists a positive constant $C$ such that $$\Vert\Delta u \Vert_\infty\leq C\Big(1+\frac1k\Vert u_n\Vert_\infty+\Vert \Delta u_n\Vert _\infty\Big).$$ 
Moreover, for each compact $K\subset \Omega$ there exists a positive constant $c(K)>0$ such that 
$$\sup_{i,j}\sup_{x\in K} |D_{ij}u(x)|\leq c.$$

\end{theorem}

\noindent The proof follows
from \cite{SmoothVariational_Brezis}, where penalty argument is applied together with a non-degeneracy argument, which we refer to Lemma 1.2 from \cite{SmoothVariational_Brezis}. 

\begin{remark}
This upper bound can be improved by applying energy gradient flow. By competing $u$ with $u_n$ in $E[u]\leq E[u_n]$, we have
$$\frac1k\Vert u-u_n \Vert_2^2\leq \Vert \nabla u_n\Vert_2^2+\Vert u_n\Vert_1.$$
By applying the penalty argument in \cite{SmoothVariational_Brezis},
we can derive that
$$\Vert\Delta u\Vert_2\leq C(\Vert u_n\Vert_{H^2}+1).  $$

\end{remark}

Now we will show the energy minimization scheme has a limit, as the time step $k \rightarrow 0$, that solves the parabolic variational inequality (\ref{eq_paravari}). We study the energy minimization scheme as in previous sections and by Proposition \ref{prop_equiv_vari}, it suffices to show the following lemma.

\begin{lemma}[Rothe's method]\label{lem_RoMet}
Recall that $k$ is the small time step and suppose for each $j=1,\cdots M$, where $M=T/k$, $u_j$ is the unique minimizer of $E_j(u)$  in $H^1_+(\Omega)$. Here $E_j(u)$ are defined similarly as in Theorem \ref{thm_mini}:
\begin{align*}
E_j(u)\coloneqq  \int_{\Omega} \frac12 |\nabla u|^2+\frac{1}{2k} (u-u_{j-1})^2+u.
\end{align*}
 Then $u=\lim_{k \rightarrow 0} u_M(x,t)$ exists and solves the parabolic variational inequality \eqref{eq_paravari}. Here $u_M$ is the associated linear interpolation defined by
 \begin{equation*} 
     u_{M}(x,t)\coloneqq(1-\theta)\cdot u_j(x)+\theta \cdot u_{j+1}(x)\ ,\ \mbox{for}\ 
     t = (j+\theta) k,\ \theta \in[0,1).
 \end{equation*} 
\end{lemma}

\begin{proof}[Proof of Lemma \ref{lem_RoMet}]
Our proof follows \cite{Kacur_RotheMethod,Rudd_Sch_vari}. First, note that $u_j$ is the unique minimizer of $E_j$ and as discussed earlier in Proposition \ref{prop_equiv_vari}, it satisfies the elliptic variational inequality \eqref{eq_elli_mini}:
\begin{equation}\label{eq_elli_vara2}
\int_\Omega \na u_j\cdot \na(v-u_j)+\frac{u_j}{k}(v-u_j)\ dx\geq \int_\Omega \left(\frac{u_{j-1}}{k}-1\right)\cdot (v-u_j)\mbox{for all } v\in H^1_+(\Omega) .
\end{equation}
Taking $v=u_{j-1}$, one can derive 
\begin{equation*}
    \langle \na u_j , \na(u_{j-1}-u_j)\rangle+\frac1k\langle u_j,u_{j-1}-u_j\rangle \geq \frac1k\langle u_{j-1} ,u_{j-1}-u_j\rangle-\langle 1, u_{j-1}-u_j\rangle.
\end{equation*}
Similarly we take $v=u_j$ in the $j-1$-th inequality:
\begin{equation*}
    \langle \na u_{j-1} ,\na( u_{j}-u_{j-1})\rangle+\frac1k\langle u_{j-1},u_{j}-u_{j-1}\rangle \geq \frac1k\langle u_{j-2},u_{j}-u_{j-1}\rangle-\langle 1, u_j-u_{j-1}\rangle.
\end{equation*}
Adding the two inequalities above it follows that
\begin{equation*}
    \frac{1}{k}\Vert u_j-u_{j-1}\Vert_2^2+\Vert \na(u_j-u_{j-1})\Vert_2^2\leq \frac1k \langle u_j-u_{j-1} , u_{j-1}-u_{j-2}\rangle.
\end{equation*}
Note that when $j=1$, we choose $v=u_0$ and hence
\begin{equation*}
    \frac{1}{k}\Vert u_1-u_0\Vert_2^2+\Vert \na(u_1-u_0)\Vert_2^2\leq \left|\langle \na u_0, \na(u_1-u_0)\rangle\right|+\left|\langle1, u_0-u_1\rangle\right|\leq \left(\Vert \Delta u_0\Vert_2+1\right)\cdot \Vert u_1-u_0\Vert_2.
\end{equation*}
Therefore we obtain that
\begin{equation*}
    \Vert \frac{u_j-u_{j-1}}{k}\Vert_2 \leq C
\end{equation*}
for any $j=1,\cdots, M$ and a positive absolute constant $C$. Note that $\dot{u_M}(t)=\frac{u_{j+1}-u_{j}}{k}$, therefore by Arzel\`a-Ascoli Theorem, $u_M(t)$ converges to some function $u$ in $C([0,T],L^2(\Omega))$. Then we can define 
$$\widetilde{u_M}(t)=u_j\ ,\ \mbox{for}\ t\in[jk,(j+1)k),$$
similar to Lemma \ref{lem_gaconv} and Remark \ref{rmk_unifconv}, $\widetilde{u_M}$ converges to the same $u$. Indeed, $\na u_M$ converges to $\na u$ weakly in $L^2((0,T),L^2(\Omega))$. As a result we can rewrite \eqref{eq_elli_vara2} as follows: for any $v\in H^1_+(\Omega)$, we have
\begin{equation}
    \langle \dot{u_M}(t), v(t)-\widetilde{u_M}(t)\rangle+\langle \na\widetilde{u_M},\na(v-\widetilde{u_M})\rangle\geq -\langle 1, v-\widetilde{u_M}\rangle,
\end{equation}
for a.e. $t\in(0,T)$. It then implies that for arbitrary $\tau_1< \tau_2$ in $[0,T]$,
\begin{equation}
    \int_{\tau_1}^{\tau_2}\langle \dot{u_M}(t), v(t)-\widetilde{u_M}(t)\rangle+\langle \na\widetilde{u_M},\na(v-\widetilde{u_M})\rangle\ dt\geq -\int_{\tau_1}^{\tau_2}\langle 1, v-\widetilde{u_M}\rangle\ dt,
\end{equation}
letting $k \to 0$, we derive the desired result
\begin{equation}
    \int_{\tau_1}^{\tau_2}\langle \dot{u}(t), v(t)-u(t)\rangle+\langle \na u,\na(v-u)\rangle\ dt\geq -\int_{\tau_1}^{\tau_2}\langle 1, v-u\rangle\ dt,
\end{equation}
for almost every $\tau_1<\tau_2$ in $[0,T]$.

\end{proof}

\subsubsection{A regularized formulation}
\label{s:regular} 

We introduce a formulation using a regularization method with parameter $\epsilon$ proposed first in \cite{rogers}. Here, we will see the convergence in regularized solutions $u_\epsilon (x,t)$ as $\epsilon \to 0$ to the other OD formulations. Its analysis is simplified since the approximating problems avoid handling the free interfaces directly. 
We include this approach for completeness. While there is theoretical insight to be gained from this formulation, it is unattractive for numerical approximation for application purposes as free interface locations are not easily identified from $\epsilon >0$ results. 
\begin{equation}\label{eq_epi_appro}
    \partial_tu_{\epsilon }=\Delta u_{\epsilon }-f_\epsilon(u_\epsilon),
\end{equation}
where 
\begin{equation}
    f_\epsilon(u_\epsilon)=\left\{
    \begin{aligned}
    1&& u_\epsilon> \epsilon  \\
\frac {u_\epsilon}{\epsilon}&& u_\epsilon\leq \epsilon ,
    \end{aligned}\right.
\end{equation}
with same initial condition $u_0(x)$. Note that $f_\epsilon(x)$ is a Lipschitz function and as a result $u_\epsilon$ exists as a smooth solution for each $\epsilon>0$ with $u_\epsilon(x,t) >0$ for all $x\in \Omega$ and $t>0$. 

We consider $u_{\epsilon_1}$ and $u_{\epsilon_2}$ with $\epsilon_1<\epsilon_2$. Denote their difference by $w=u_{\epsilon_1}-u_{\epsilon_2}$, then
\begin{align*}
    \partial_t w-\Delta w=-f_{\epsilon_1}(u_{\epsilon_1})+f_{\epsilon_2}(u_{\epsilon_2}).
\end{align*}
Note that,
\begin{equation}
  -f_{\epsilon_1}(u_{\epsilon_1})+f_{\epsilon_2}(u_{\epsilon_2})=  \left\{
  \begin{aligned}
  0\ ,\ &&\mbox{if}\ u_{\epsilon_1}>\epsilon_1\ ,\ u_{\epsilon_2}>\epsilon_2\\
  -\frac{u_{\epsilon_1}}{\epsilon_1}+1\ ,\ &&\mbox{if}\ u_{\epsilon_1}\leq \epsilon_1\ ,\ u_{\epsilon_2}>\epsilon_2\\
  -1+\frac{u_{\epsilon_2}}{\epsilon_2}\ ,\ &&\mbox{if}\ u_{\epsilon_1}> \epsilon_1\ ,\ u_{\epsilon_2}\leq\epsilon_2\\
  -\frac{u_{\epsilon_1}}{\epsilon_1}+\frac{u_{\epsilon_2}}{\epsilon_2}\ ,\ &&\mbox{if}\ u_{\epsilon_1}\leq \epsilon_1\ ,\ u_{\epsilon_2}\leq\epsilon_2
  \end{aligned}
  \right. .
\end{equation}
We observe that
\begin{align*}
    -1+\frac{u_{\epsilon_2}}{\epsilon_2}\leq 0 \ ,\ &&\mbox{if}\ u_{\epsilon_1}> \epsilon_1\ ,\ u_{\epsilon_2}\leq\epsilon_2\\
    -\frac{u_{\epsilon_1}}{\epsilon_1}+\frac{u_{\epsilon_2}}{\epsilon_2}=-\frac{w}{\epsilon_1}+u_{\epsilon_2}\cdot\left(\frac{1}{\epsilon_2}-\frac{1}{\epsilon_1}\right)\ ,\ &&\mbox{if}\ u_{\epsilon_1}\leq \epsilon_1\ ,\ u_{\epsilon_2}\leq\epsilon_2\\
    u_{\epsilon_1}<u_{\epsilon_2}\ ,\ &&\mbox{if}\ u_{\epsilon_1}\leq \epsilon_1\ ,\ u_{\epsilon_2}>\epsilon_2;
\end{align*}
so if we assume the maximal value of $w$ is achieved at $(x_0,t_0)$ with $x_0\in\Omega$ and $t_0>0$, then $w(x_0,t_0)>0$, $\partial_tw(x_0,t_0)=0$ and $\Delta w(x_0,t_0)\le 0$. Then the standard maximum principle gives a partial result of the following statement:
\begin{theorem}
\label{th:epsilon} 
Suppose a sequence of classical functions $\{u_{\epsilon}\}$ solve \eqref{eq_epi_appro}, then $u_{\epsilon}$ is monotonically decreasing as $\epsilon $  decreases to 0. Moreover, the limiting function $$\lim_{\epsilon\to0}u_{\epsilon}=u$$
holds pointwisely. This limiting function $u$ solves the variational inequality \eqref{eq_paravari}.
\end{theorem}
The details of the proof can be found in Appendix~\ref{s:eproof}. An alternate convergence statement and proof is given in Appendix~\ref{s:AugL}.

\subsection{Conjecture on the general dynamics in 1D} 

We make the following plausible conjecture for the general dynamics (including topological changes) of the Cauchy problem in 1D with initial conditions $u_0(x) \in H^1_+(\Omega)$ with compact support. Here, we consider the problem for all space rather than half space with a no flux condition at $x=0$.  
\begin{conjecture} 
\label{con:S}
Assume $u_0$ has a finite ${\cal S}(0)$ where 
${\cal S}(t)$ counts the number of free boundary points:
\[
{\cal S}(t) = \left\{ x : u(x,t) = 0 \mbox{\ and $u(y,t)> 0$ for some $y$ in every neighbourhood of $x$} \right\}.
\]  
Then 
\begin{description}
\item[(i)] ${\cal S} (t)$ is finite for every $t>0$.
\item[(ii)] There exists a finite increasing sequence of times $t_j$, $j= 0, \ldots ,M$ with $t_0 = 0$ and  $\operatorname{card } {\cal S} (t):=n_j$ constant on every interval $(t_j,t_{j+1})$ and $u \equiv 0$ for $t \geq t_M$. 
\item[(iii)] ${\cal S} (t) = \{ s_1(t), s_2(t), \ldots s_{n_j} (t) \}$ for $s_l(t)$ smooth on $(t_j,t_{j+1})$.
\item[(iii)] $u(x,t)$ is $C^1$ for $t > 0$ and $C^\infty$ except at free boundary points. 
\end{description}
\end{conjecture}
Recent related results have been shown for the Stefan problem \cite{figalli}. Similar analysis of the OD problem is complicated by the reaction term that allows the formation of new zones of constraint ($u \equiv 0$).

\section{Numerical Approximation} 
\label{s:numerical} 

We consider two numerical methods. The first, suitable for 1D dynamics without topological change, is based on the mapped domain formulation described in Section~\ref{s:yproblem}. The second, suitable for dynamics in higher dimensions including topological change, is based on our new gradient formulation described in Section~\ref{s:gradient}. We prove convergence of this scheme. 

\subsection{Mapped domain ($Y$ formulation) method} 
\label{s:ymethod} 

We consider the discretization of the mapped domain formulation (\ref{eq_yproblem}) in space using cell centred finite differences. We first discretize in space, leaving time continuous (known as a Method of Lines -- MoL -- discretization) with approximations $u^j(t) \approx u((j-1/2)h,t)$, $j = 1 \ldots N$ where $h$ is the uniform grid spacing with $N$ subintervals of $y \in [0,1]$. The interface location $s(t)$ is approximated by $S(t)$. 

\begin{figure}
\centerline{
\includegraphics[width=10cm]{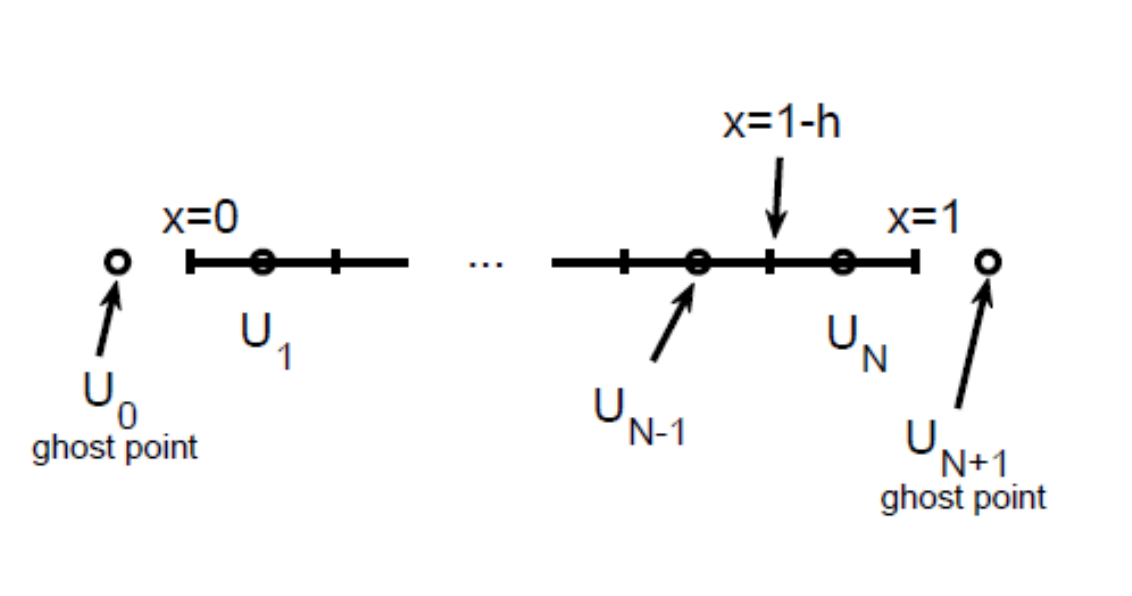}}
\caption{\label{f:ghost} Cell centered finite difference spatial approximation of the 1D mapped domain formulation}  
\end{figure}

Boundary conditions are implemented using ghost points \cite{ghost} $u^0(t) \approx u(-h/2,t)$ and $u^{N+1}(t) \approx u(1+h/2,t)$ depicted in Figure~\ref{f:ghost}. 
Boundary conditions at $y=1$ are implemented using second order averages and differences:
\begin{eqnarray}
(u^{N+1} + u^N)/2 & = & 0 \label{eq:bc1} \\
(u^{N+1} -u^N)/h & = 0 \label{eq:bc2}
\end{eqnarray}
which implies that $u^{N} = u^{N+1} = 0$. The no-flux boundary condition at $y=0$ is approximated similarly. The MoL discretization for the interior equations is 
\begin{equation}
\label{eq:interior}
D_2 u^j + S \dot{S} y D_1 u^j -S^2 \dot{u}^j - S^2  = 0. 
\end{equation} 
where $D_2$ and $D_1$ are the standard centered second order finite difference operators. 
The system (\ref{eq:bc1},\ref{eq:bc2},\ref{eq:interior}) is a Differential Algebraic Equation (DAE) \cite{DAE} and has index one. For computational results, we use Implicit (Backward) Euler time stepping with Newton iterations for the resulting nonlinear system at each time step.  In a computational study, we observe errors of size $O(h^2)+ O(k)$ where $k$ is the time step, as expected for a second order spatial and first order temporal discretization. 

\begin{remark}
\label{rem:ymethod} The convergence of the method has not been proved. The missing direct analysis discussed in Remark~\ref{rem:yproblem} could give insight. 
\end{remark}

\subsubsection{Computational results}

Examples of the dynamics computed with the DAE formulation in the mapped region are shown in Figure~\ref{f:ODex1}. The left figure shows the solution with initial conditions $u_0(x) = (1-x)^2/2$ for $x \in [0,1]$ considered often in the literature. It is the steady state of the problem forced with flux condition $u_x = -1$ at $x=0$ \cite{Crank}. In this solution, $s(t)$ moves monotonically to the left. The solution in this formulation ends when $s(T)=0$ ($u \equiv 0$). A specialized method in this general framework was developed in \cite{oxygen} to accurately compute both the solution and the end time $T$ of the dynamics. Our mapped formulation breaks down as $t \rightarrow T$. The right computation of Figure~\ref{f:ODex1} with initial conditions $u_0 = -x^4 +3x^3-5x^2/2 + 1/2$ for $x \in [0,1]$ shows that $s(t)$ does not have to be monotone decreasing.  
Here, $s(t)$ initially moves to the right driven by diffusion and then to the left as $u$ values decrease due to the consumption term. 

\begin{figure}
\centerline{\includegraphics[width=8cm]{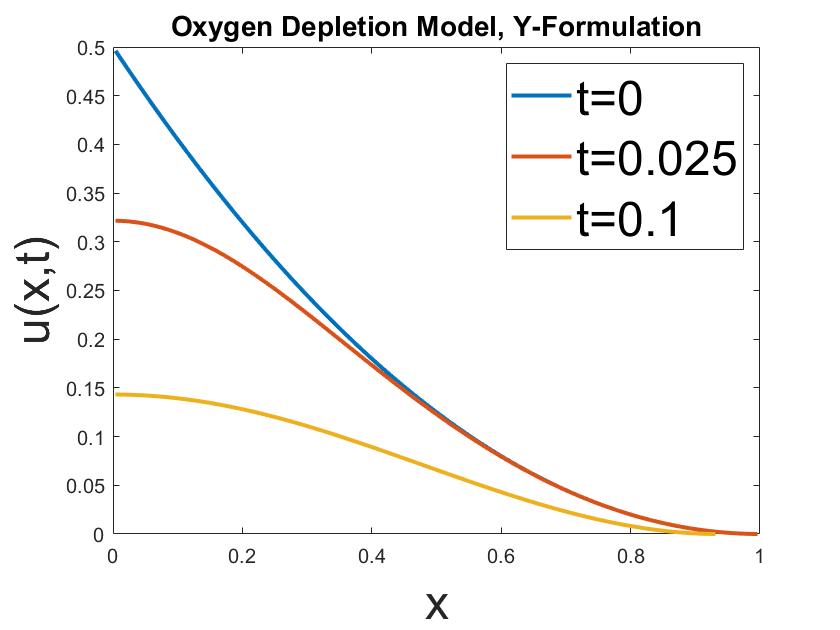}
\includegraphics[width=8cm]{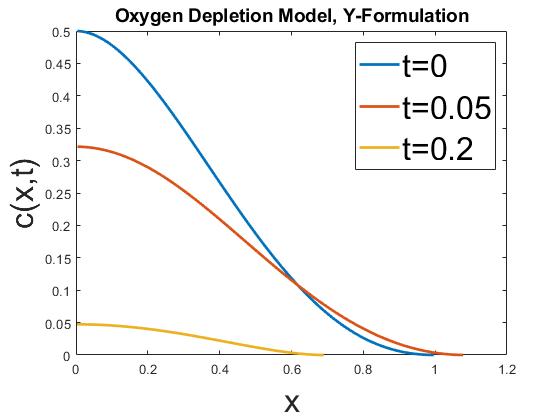}
}
\caption{\label{f:ODex1} 1D solutions of the OD problem without topological change. 
Left: initial conditions $u_0(x) = (1-x)^2/2$ (standard problem in the literature), $s(t)$ decreases monotonically. Right: initial conditions $u_0(x) = -x^4 +3x^3-5x^2/2 + 1/2$,  $s(t)$ initially moves to the right driven by diffusion and then to the left as $u$ values decrease due to the consumption term.} 
\end{figure}

\subsection{Gradient flow method}
\label{s:brians} 

In this section, we continue the discretization of the gradient flow formulation from Section~\ref{s:gradient} and discretize in space with $u^{i,j}_n \approx u(ih,jh,nk)$. We consider the discretization in two spatial dimension for ease of presentation but the argument extends to other dimensions. The energy minimization problem (\ref{eq_elli_mini}) is approximated by the discrete minimization of 
\begin{equation} \label{eq_dis_mini}
    E_{n+1}^N=\frac{h^2}{2}\sum_{i,j=0}^{N-1} (\left( \frac{u^{i+1,j}-u^{i,j}}{h}\right)^2+\left(\frac{u^{i,j+1}-u^{i,j}}{h}\right)^2)+h^2\sum_{i,j=0}^N\left(\frac{1}{2k} \left(u^{i,j}-u_n^{i,j}\right)^2+u^{i,j}\right),
\end{equation}
where $N$ the number of grid points with $N=1/h$, assuming without loss of generality that all positive values of $u$ are captured in $(0,1)\times(0,1)$. We solve this minimization problem subject to all non-negative discrete data $$\vec{u}_{n+1}^N\coloneqq (u^{1,1}_{n+1},u^{2,1}_{n+1},\cdots,u^{N,1}_{n+1},u^{1,2}_{n+1},\cdots, u^{N,N}_{n+1}).$$ This is a convex, quadratic minimization problem with linear, inequality constraints and so has a unique global minimum. We show below that the solution to the discrete optimization problem converges to the OD solutions as $h,k \rightarrow 0$. In Section~\ref{s:dopt} we discuss the technique we use to solve the discrete optimization problem. 
Denoting $M=T/k$, we use $\{\vec{u}_{n}\}_{n=1}$ to define an approximate solution:
\begin{equation}\label{eq_approx_sol}
    u_{N,M}(x,t)\coloneqq\left\{\begin{aligned}
  && (1-t)\cdot u_0(x)+t\cdot u_1^N(x) &&\mbox{for}\ t\in[0,k)&&\\&&\cdots&&\\
  &&(1-t)\cdot u_m^N(x)+t\cdot u_{m+1}^N(x)  &&\mbox{for}\ t\in[mk,(m+1)k)&&\\&&\cdots&&\\
  &&(1-t)\cdot u_{M-1}^N(x) +t\cdot u_{M}^N(x) &&\mbox{for}\ t\in[(M-1)k,T],&&
    \end{aligned}      \right.
\end{equation}
where $u_0(x)$ is the initial condition and for $1\leq m\leq M$ and $u_m^N$ is the linear (bilinear) approximation function. In particular when $x=(x_1,x_2)\in [ih,(i+1)h)\times [jh,(j+1)h)$ we define $u_m^N$ as follows:
\begin{equation}\label{3.6}\begin{aligned}
    u_m^N(x_1,x_2)=&h^{-2}\Big(u_m^{i,j}((i+1)h-x_1)((j+1)h-x_2)+u_m^{i,j+1}((i+1)h-x_1)(x_2-jh)\\
    &+u_m^{i+1,j}(x_1-ih)((j+1)h-x_2)+u_m^{i+1,j+1}(x_1-ih)(x_2-jh)\Big).
    \end{aligned}
\end{equation}
Therefore we have when $(x_1,x_2)\in (ih,(i+1)h)\times (jh,(j+1)h)$ 
\begin{align*}
&\pa_1 u_m^N(x_1,x_2)=h^{-2}\left((u^{i+1,j}_m-u^{i,j}_m)((j+1)h-x_2)+(u^{i+1,j+1}_m-u^{i,j+1}_m)(x_2-jh)\right),\\
&\pa_2 u_m^N(x_1,x_2)=h^{-2}\left((u^{i,j+1}_m-u^{i,j}_m)((i+1)h-x_1)+(u^{i+1,j+1}_m-u^{i+1,j}_m)(x_1-ih)\right).
\end{align*}Therefore, by direct computation we obtain that
\begin{align}\label{3.7}
\int_{jh}^{(j+1)h}\int_{ih}^{(i+1)h}(\pa_1 u_m^N)^2\ dx_1dx_2=&\frac13 (u_m^{i+1,j}-u^{i,j}_m)^2
+\frac13(u_m^{i+1,j+1}-u_m^{i,j+1})^2\\&+\frac13(u_m^{i+1,j+1}-u_m^{i,j+1})(u_m^{i+1,j}-u^{i,j}_m).\notag
\end{align}
Moreover the pointwise limit 
\begin{equation}
\label{eq:uj} 
u_m=\lim_{N\to\infty}u_m^N 
\end{equation}
exists. For convenience we also define
\begin{align}
\widetilde{u_m^N}(x)\coloneqq    \sum_{i,j=1}^{N} u^{i,j}_m\cdot\chi_{((i-1)h,ih)}(x_1)\chi_{((j-1)h,jh)}(x_2),
\end{align}
where $\chi_{I}$ is the characteristic function on the interval I. We also define
 \begin{equation} \label{eq_time_appro}
     u_{M}(x,t)\coloneqq(1-t)\cdot u_m(x)+t\cdot u_{m+1}(x)\ ,\ \mbox{for}\ t\in[mk,(m+1)k).
 \end{equation}
 \begin{remark}\label{rmk_unifconv}
 We observe that both approximations $u_m^N(x)$ and $\widetilde{u_m^N}(x)$ will converge to the same limit in $L^2(\Omega)$ as $N\to\infty$ and similarly for $u_{N,M}(x,t)$ and $\widetilde{u_{N,M}}(x,t)$ as $N\to\infty$. Arzel\`a-Ascoli theorem and the finite energy assumption then imply the uniform convergence as in Lemma~\ref{lem_RoMet}.
 \end{remark}

\begin{theorem}\label{thm_discrete_conv}
Suppose $\vec{u}_{n+1}^N$ solves the discrete minimization problem \eqref{eq_dis_mini}, then 
\begin{equation*}
    u(x,t)=\lim_{M\to\infty}\lim_{N\to\infty}u_{N,M}
\end{equation*}
with $h = 1/N$ and $k=T/M$ exists in $\mathcal{J}$ and $u$ is the solution to the variational inequality \eqref{eq_paravari} that is
\begin{align*}
    \int_0^t\int_\Omega u_t\cdot (v-u)+\int_0^t\int_\Omega \na u\cdot \na(v-u)\geq \int_0^t\int_\Omega u-v;\ \mbox{for all}\ v\in\mathcal{J},\ \mbox{a.e.}\ t\in(0,T). 
\end{align*}

\end{theorem}

\noindent The proof relies on two lemmas. To start with, we give definitions of gamma convergence of energy functionals shown in Lemma \ref{lem_gaconv} as given in \cite{GaConv}:

\begin{defn}[Gamma convergence]
We say that the sequence of functionals $\{\mathcal{E}_l\} :X \to \R \cup \{-\infty, +\infty\}$
where $X$ is a metric space, $\Gamma$-converges to $\mathcal{E}$ if the following conditions are satisfied:
\begin{description}
      \item[i] whenever $x_l\to x$, $\mathcal{E}(x)\leq \liminf_l \mathcal{E}_l(x_l)$;
    \item[ii] for any $x\in X$, there exists $x_l\to x$ in $X$ such that $\limsup_l \mathcal{E}_l(x_l)\leq \mathcal{E}(x)$.
\end{description}
\end{defn}
\noindent The following is a relevant property of $\Gamma$-convergence:
\begin{prop}\label{prop_GaMin}
Given a metric space $X$ and suppose a sequence of functionals defined $\mathcal{E}_l$ defined in $X$ $\Gamma$-converges to $\cE$. Assume that for each $l$, $x_l$ is a minimizer of $\mathcal{E}_l$, and if $\bar{x}$ is a cluster point of $\{ x_l \}$, then $\bar{x}$ is a minimizer of $\cE$.
\end{prop}
\noindent We refer the proof to \cite{GaConv} (Corollary 7.20.). We then consider the following energy functional:
\[
E_{n+1}(u_{n+1})=\int_\Omega \frac12 |\na u_{n+1}|^2+\frac{1}{2k} (u_{n+1}-u_n)^2+u_{n+1}
\]
where $u_{n+1}(x)$ is defined in (\ref{eq:uj}). 

\begin{lemma}[Gamma convergence of discrete functionals]\label{lem_gaconv}
For each $n$, $E_{n+1}^N$ $\Gamma$-converges to $E_{n+1}$ as $N\to \infty$ or equivalently $h\to 0$ in $L^2(\Omega)$. 
\end{lemma}

\begin{proof}[Proof of Lemma \ref{lem_gaconv}]
We follow the proof in \cite{GaConv}.

\noindent To show (i): let $u_{n+1}^N\in L^2(\Omega)$ such that $\liminf E_{n+1}^N(u_{n+1}^N)<+\infty$ and therefore there exists a subsequence $u_{n+1}^{N_l}$  such that $\lim E_{n+1}^{N_l}(u_{n+1}^{N_l})=\liminf E_{n+1}^N(u_{n+1}^N)$. For each $l$, there exists a mesh of grid points and a vector $\vec{u}_{n+1}^{N_l}$($\in\R^{N_l+2}\times \R^{N_l+2}$ in the 2D Neumann boundary condition case) such that the corresponding $u_{n+1}^{N_l}(x)$ is defined in \eqref{eq_approx_sol}-\eqref{3.6}. Then by the previous Remark~\ref{rmk_unifconv}, both $u_{n+1}^{N_l}$ and $\widetilde{u_{n+1}^{N_l}}$ converge to the same limit $u$ in $L^2$. By \eqref{3.7} we also have
\begin{align*}
h^2\sum_{i,j=0}^{N_l-1} (\left( \frac{u_{n+1}^{i+1,j}-u_{n+1}^{i,j}}{h}\right)^2+\left(\frac{u_{n+1}^{i,j+1}-u_{n+1}^{i,j}}{h}\right)^2)\ge\int_{\Omega}|\na u_{n+1}^{N_l}|^2.
\end{align*}
Thus
\begin{equation*}
    \int_\Omega |\na u|^2\leq\lim_l \int_\Omega |\na u_{n+1}^{N_l}|^2\leq \liminf_N h^2\sum_{i,j=0}^{N-1} (\left( \frac{u^{i+1,j}_{n+1}-u^{i,j}_{n+1}}{h}\right)^2+\left(\frac{u^{i,j+1}_{n+1}-u^{i,j}_{n+1}}{h}\right)^2).
\end{equation*}
On the other hand, 
\begin{equation*}
    h^2\cdot\sum_{i,j=0}^{N_l}\left(\frac{1}{2k} \left(u^{i,j}_{n+1}-u_n^{i,j}\right)^2+u_{n+1}^{i,j}\right)=\int_\Omega \frac{1}{2k}(\widetilde{u_{n+1}^{N_l}}-\widetilde{u_{n}^{N_l}})^2+\widetilde{u_{n+1}^{N_l}}.
\end{equation*}
Applying the uniform convergence we obtain that
\begin{equation*}
    \int_\Omega \frac{1}{2k}(u-u_n)^2+u\leq\lim_k \int_\Omega \frac{1}{2k}(\widetilde{u_{n+1}^{N_l}}-\widetilde{u_{n}^{N_l}})^2+\widetilde{u_{n+1}^{N_l}}\leq \liminf_N h^2\cdot\sum_{i,j=0}^N\left(\frac{1}{2k} \left(u_{n+1}^{i,j}-u_n^{i,j}\right)^2+u_{n+1}^{i,j}\right).
\end{equation*}
These two estimates lead to $E_{n+1}(u)\leq\liminf E_{n+1}^N(u_{n+1}^N)$.

It remains to prove (ii): suppose $u\in L^2$ with $E_{n+1}(u)<+\infty$, so $u\in H^1$. We then define $u^{i,j}_{n+1}\coloneqq u(\frac{i}{N},\frac{j}{N})$ which defines the vector $\vec{u}_{n+1}^N$ with the piecewise linear (bilinear) approximation $u_{n+1}^N(x)$ and piecewise constant approximation $\widetilde{u_{n+1}^N}(x)$. By the finite energy assumption, Arzel\`a-Ascoli theorem then guarantee the uniform convergence as in Remark \ref{rmk_unifconv}. It then follows that
\begin{equation*}
    \limsup E_{n+1}^N(u^N_{n+1})\leq E_{n+1}(u).
\end{equation*}
The 1D and 3D cases can be treated similarly.
\end{proof}
\noindent As a result of Lemma \ref{lem_gaconv} and Proposition \ref{prop_GaMin}, we obtain the following corollary immediately:
\begin{corollary}
Suppose $u^N_{n+1}$ are minimizers of $E_{n+1}^N$ then $u_{n+1}^N$ converges to a function $u_{n+1}$ in $L^2(\Omega)$ up to a subsequence as $h\to 0$ and such $u_{n+1}$ is the minimizer of $E_{n+1}$.  
\end{corollary}
\noindent Now that $u_m\coloneqq\lim_{N}u_m^N$ is the minimizer of the continuous functional $E_{m}$ for $m=1,\cdots, M$; it remains to show that $u(x,t)=\lim_{M\to\infty}u_M(x,t)$ solves the variational inequality \eqref{eq_paravari}. Recalling the Rothe's Method (Lemma \ref{lem_RoMet}) and combining results of Lemma \ref{lem_gaconv} and Lemma \ref{lem_RoMet}, we therefore complete the proof of Theorem \ref{thm_discrete_conv}.

 \subsubsection{Discrete Optimization Scheme}
 \label{s:dopt}
 
We consider the details of the discrete optimization problem (\ref{eq_dis_mini}) and present the scheme in the 2D case. (Note that this scheme holds in 1D and 3D similarly.) 
The corresponding Lagrangian problem is 
 \begin{align*}
    & -\Delta_h u +\frac{u}{k}+\lambda = \frac{u_n}{k}-1,&&\\
    & \lambda_{(j_1,j_2)} < 0,\  u^{(j_1,j_2)} = 0 &&\forall (j_1,j_2)\in J\\
    &\lambda_{(i_1,i_2)}=0,\ u^{(i_1,i_2)} \geq 0 &&\forall (i_1,i_2)\in I,
 \end{align*}
where $I$ and $J$ are a disjoint partition of the grid points and $\Delta_h$ is the finite difference Laplacian. The $u$ in the problem is the grid vector at the next time step $u_{n+1}$. The partitions divide those points $J$ where the values are at the constraint and those points $I$ (``$I$'' for inactive constraint) with positive solution values where the corresponding derivative of $E^N$ must be zero.
The method is an active set method, where the sets $J$ and $I$ are updated iteratively at each time step. 
Note that $\lambda_{(j_1,j_2)} < 0$ for $(j_1,j_2) \in J$ corresponds to $\partial E^N/\partial u^{(j_1,j_2)} >0$, a necessary and sufficient condition for optimality (the KKT conditions \cite{KKT}).  
There are many techniques available to solve such quadratic optimization problems with linear inequality constraints. We take advantage of the simple structure of the problem and the fact that there is little change in the index sets from one time step to the next in the following algorithm. It is an iterative algorithm with vectors $u^{(m)}$, $\lambda^{(m)}$ at each iteration. The matrix $A = I/k - \Delta_h$, where $I$ is the identity. 

 \paragraph{Algorithm}
 \begin{description}
       \item[Step 1] Initialize $u^{(0)} \geq 0$ (component-wise), $\lambda^{(0)} = \min\{0, \frac{u_n}{k}-1-Au^{(0)} \}$. Set $m=0$. Repeat steps 2-5 until the convergence criteria in step 3 is reached. 
       \item[Step 2] Construct the index sets
        \begin{equation*}
        \begin{aligned}
           J^{(m)}=\{(j_1,j_2):\la^{(m),(j_1,j_2)}<0\},\\
            I^{(m)}=\{(j_1,j_2):\la^{(m),(j_1,j_2)}=0\}. 
        \end{aligned}
        \end{equation*}
        For any $(i_1,i_2) \in I^{(m)}$ such that $u^{(m),(i_1,i_2)}< 0$, move $(i_1,i_2)$ to $J^{(m)}$. 
         
         \item[Step 3] If $J^{(m)}=J^{(m-1)}$, the solution $u=u^{(m)}$. Stop. 
         \item [Step 4] Solve for $u^{(m+1)}$ and $\lambda$ using 
         \begin{align*}
             &Au^{(m+1)}+\la=\frac{u_n}{k}-1,\\
             &\la=0 \ \mbox{on}\ I^{(m)},\\
             &u^{(m+1)}=0\ \mbox{on}\ J^{(m)}.
         \end{align*}
         This is equivalent to solving sequentially for  $(u^{(m+1)},\la)$ that satisfy
       \begin{align*}
           &A_{II} u^{(m+1)}_I=(\frac{u_n}{k}-1)_I,\\
           &u^{(m+1)}_J=0,\\
           &\lambda =\frac{u_n}{k}-1-A u^{(m+1)}.
       \end{align*}
Here vector subscripts $I$ and $J$ give the sub-vectors with those components and $A_{II}$ is the block of the matrix A corresponding to the $I$ components. 
\item[Step 5] Update $\lambda^{(m+1)}=\min\{0,\lambda \}$. Increment $m$. 
\end{description}

\begin{theorem}
Let $0\leq u^{(0)} \leq u$ (component-wise). The algorithm above converges in finitely many steps.
\end{theorem}
  
 \begin{proof}
A proof is found following closely the ideas from \cite{AugLag_actset} for a similar approach to the elliptic obstacle problem. 
Monotone behaviour in the index sets $I^{(m)}$ is shown and since $N$ is finite, the algorithm converges in finite steps. 
Use is made of the properties that the sub-matrix $A_{II}^{-1}$ has positive entries ($A_{II}$ is monotone) and $A_{IJ}$ has non-positive entries (values zero or $-1/h^2$) for any index sets $I$ and $J$.  
\end{proof}
 
 
 
\begin{remark}
\label{rem:dopt}
While the proof of iteration convergence above is limited to starting conditions $0\leq u^{(0)} \leq u$, we implement the method with $u^{(0)} = u_n$ and starting index sets from the converged iterations at time step $n$. This initialization falls out of the scope of the analysis but works well (no failures, few iterations) in practice. 
\end{remark} 

\begin{remark}
\label{rem:active} 
Similar index (active set) iteration methods have been used in capturing methods for other implicit boundary value problems. Two of these are discussed in Section~\ref{s:other}. A general theory for the convergence of these iteration strategies is not known, but they can perform well in practice. 
\end{remark} 

\subsubsection{Numerical Results} 
\label{s:numerical2} 

We show results in 1D with topological change in Figure~\ref{f:ODex2}. Initial conditions are 
\[
u_0(x) = \left\{
\begin{array}{ll}  
((1/3-x).^2+0.05)/((5/12)^2+0.05)/16 & \mbox{\ $x \in [0,3/4]$} \\
(1-x)^2 & \mbox{\ $x \in [3/4,1]$} \\
0 & \mbox{\ $x \geq 1$} 
\end{array} 
\right. 
\]

A 2D example is shown in Figure~\ref{f:ODex3}. This example has more complicated topological changes described in the figure caption. Based on evidence from other computations, the limiting circular shape is generic.  

\begin{figure}
\centerline{
\includegraphics[width=8cm]{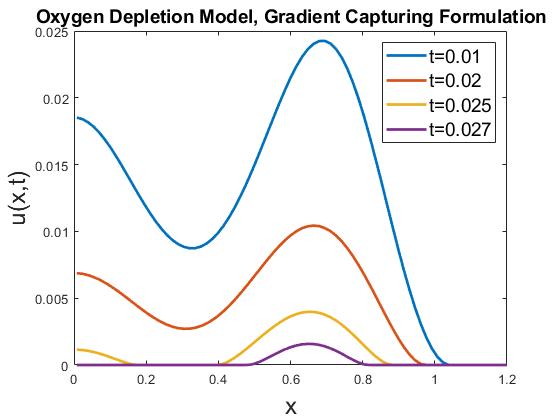}
}
\caption{\label{f:ODex2} 1D solution of the OD problem with topological changes with the gradient flow method. Initial conditions are given in Section~\ref{s:numerical2}.}
\end{figure}
 
\begin{figure}
\centerline{\includegraphics[width=6cm]{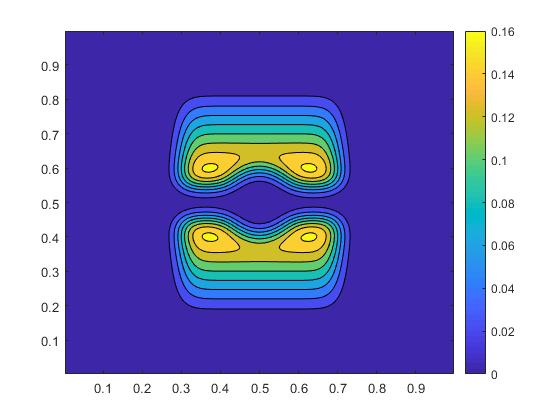}
	\includegraphics[width=6cm]{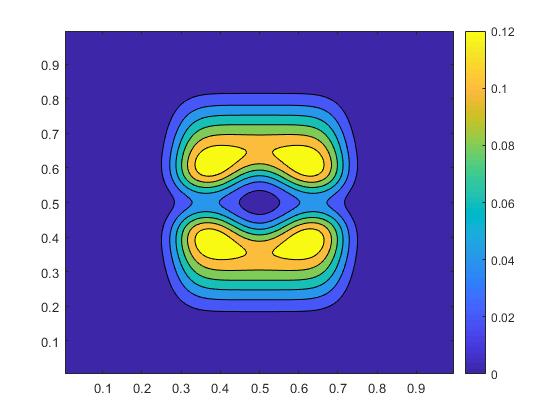}}
\centerline{\includegraphics[width=6cm]{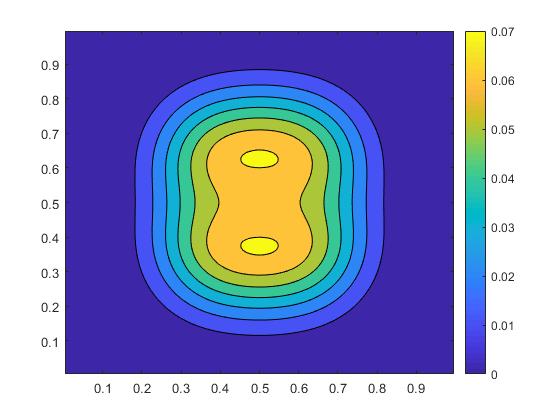}
		\includegraphics[width=6cm]{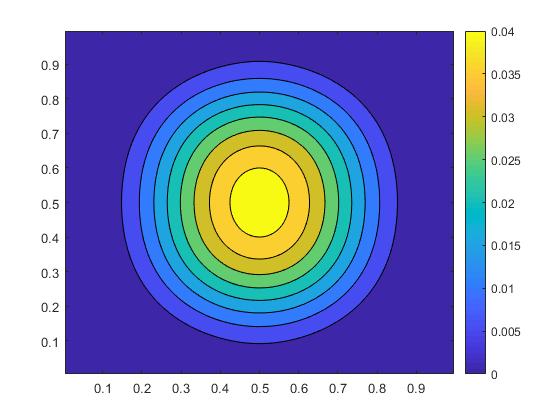}}
\caption{\label{f:ODex3}A 2D solution of the OD problem with topological change computed with the gradient flow method. Solution contours are shown with time increasing from top left to bottom right. Overall solution levels decrease in time with depletion. Two topological changes occur. The first is the merger of the two disjoint sets of $u > 0$, the second is the disappearance of the centre $u=0$ set.} 
\end{figure}


\section{Other Implicit Free Boundary Value Problems}
\label{s:other} 

\subsection{A biharmonic problem} 

The OD problem is the simplest second order implicit free boundary problem. The simplest fourth order problem is the following biharmonic problem shown in 1D for $u(x,t)$:
\[
u_t = -u_{xxxx} -1 
\]
with conditions $u = 0$, $u_x=0$, and $u_{xxx} = 0$ at the implicitly defined free boundary $x=s(t)$ and $u\equiv 0$ for $x>s(t)$. This can be derived from the scaled, linear, viscoelastic motion of a beam above a flat, rigid surface. Note that another boundary value problem occurs if $u_{xxx} = 0$ is replaced by $u_{xx} = 0$. However, the third order condition is correct for this application \cite{biharmonic} and also gives the gradient flow structure described below. 

We consider the time discretization of this problem as in Section~\ref{s:gradient} and see that it is a discrete $L_2$ gradient flow on the energy 
\[
\cE^{n}\coloneqq \int\frac12 |\Delta u_{n}|^2+u_{n}
\]
with $u_n \in H_+^2$. We form a fully discrete scheme as was done in Section~\ref{s:brians} and compute the discrete optimization at each time step using index iterations as described in Section~\ref{s:dopt}. The convergence of the method follows the same ideas as presented for the OD problem. Some computational results are shown in Figure~\ref{f:bih}. 

\begin{figure}
\centerline{
\includegraphics[width=12cm]{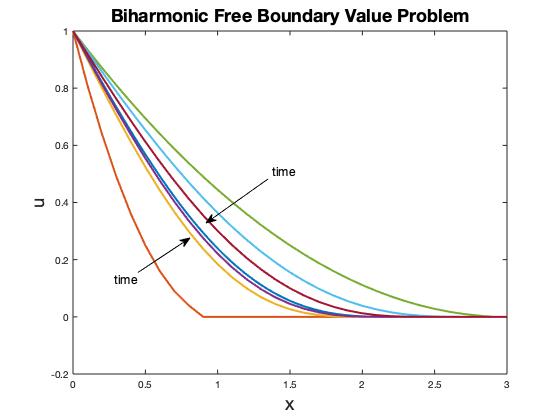}}
\caption{Two computations at three times each for the biharmonic free boundary value problem with physical boundary conditions $u(0) = 1$, $u_{xx} = 0$ approaching the analytic steady state solution shown in dark blue.
\label{f:bih}}
\end{figure}

\begin{remark}
\label{r:bihobs}
There has been considerable mathematical interest in the elliptic obstacle problem as discussed in the introduction. This is the steady state of the OD problem with nonzero physical boundary conditions. The steady state of the biharmonic problem (in higher dimensions) described in this section would also be mathematically interesting. Its analysis would be complicated by the lack of a maximum principle. 
\end{remark}

\subsection{Vector problems} 
\label{s:vector} 

The free boundary in complex fluids with yield stress is of implicit type and is well studied \cite{wachs}. Numerical approaches include regularization (increased viscosity in the unyielded region) and an Augmented Lagrangian approach to the non-smooth optimization problem that comes from a discretization of a variational inequality formulation. The literature on this problem is focussed on capturing the unyielded region rather than considering the free boundary directly. 

Implicit free boundaries in porous media flow can occur when phase change is present. Boundaries between dry and two-phase (where there is liquid and vapour present) regions were studied in \cite{huax,M2map}. The work in \cite{M2map} had important implications to simulations of water management in fuel cells. However, many theoretical questions were left unanswered and this became the motivation of the corresponding author to attempt the current work. 

We present below a class of implicit free boundary value problems that generalizes the OD problem. The problems are presented in 1D with a single free boundary at $x=s(t)$ with ${\bf u}^l(x,t)$ having $n$ components for $x < s(t)$ and ${\bf u}^r(x,t)$ having $m$ components for $x > s(t)$. Near the interface we take 
\[
{\bf u}^*_t = D^*{\bf u}^*_{xx} + {\bf a}^*
\]
for $* \in \{l,r\}$, $D^*$ positive diagonal matrices, and ${\bf a}^*$ constant vectors. At the boundary, we take 
\[
B 
\left[ \begin{array}{c} {\bf u}^l \\ {\bf u}^l_x \\ {\bf u}^r \\ {\bf u}^r_x \end{array} \right]
= {\bf 0}
\]
where $B$ is an $(m+n+1) \times (2m + 2n)$ matrix of full rank. This class can be reached from a wider class by taking affine combinations of solution components and $x$, and as an approximation of some nonlinear problems. A problem statement can be made by adding far field conditions, $n$ on the left and $m$ on the right. With these far field conditions we label the class as {\em n+m} implicit free boundary value problems. The OD problem is the only well defined example of the 1+0 class.  The model in \cite{M2map} is of class 2+2, although one of the components has degenerate diffusion at the free boundary. 

There are several open questions related to problems of this type motivated by the current work on the OD problem. Which lead to well defined problems? (this could depend on the sign of entries of $\bf a$ as discussed in Remark~\ref{rem:yproblem}). Which have gradient flow or variational inequality structure? Which allow a capturing formulation with index iteration similar to that described in Section~\ref{s:dopt}? (true of the model in \cite{M2map}). 

\section{Summary} 

This work summarizes the ways the Oxygen Depletion problem has been considered in the literature: with interfaces to be tracked, captured, or found as a limit of regularized problems. We fill in a gap in the list of formulations, showing that the OD problem can be considered as a gradient flow with constraint. A new numerical capturing method based on the  gradient flow formulation is proposed and a convergence proof given. The equivalence of all formulations is shown. A biharmonic implicit free boundary value problem and a class of vector problems are introduced. 

Several open problems have been presented in the work and are summarized here:
\begin{itemize}
\item The regularity of boundary point positions in 1D (Conjecture~\ref{con:S}) and higher dimensions. 
\item A direct analysis of the mapped domain formulation discussed in Remark~\ref{rem:yproblem} that would possibly extend to a convergence proof of its numerical approximation (Section~\ref{s:ymethod}) and an understanding of the general class of vector problems in Section~\ref{s:vector}. 
\item A convergence proof for the index iteration (active set) approach that has been successful in discretizations of implicit free boundary value problems, see Remarks~\ref{rem:dopt} and~\ref{rem:active}. 
\item A study of the biharmonic obstacle problem discussed in Remark~\ref{r:bihobs}. 
\item An understanding of the general class of vector problems introduced in Section~\ref{s:vector}.  
\end{itemize}
We hope the reader will find some of these problems of interest.

\section*{Acknowledgement}
Cheng is partially supported by the International Doctoral Fellowship (IDF) provided by the University of British Columbia and Shanghai "Super Postdoc" Incentive Plan. Fu is supported by a PhD Fellowship from the South University of Science and Technology. Wetton is supported by an NSERC Canada research grant. We thank Gwynn Elfring for pointing out the literature for the boundary conditions for the viscoelastic beam problem.  

\appendix

\section{Proof of Proposition~\ref{th:minimizer}}
\label{s:prop}

With help of the minimality of $u$, we consider a competing function $u+\varepsilon \phi$ where $\phi$ is an arbitrary smooth function that is compactly supported inside $\Omega$. By the definition of $\tilde{E}[u]$, it follows that 
\begin{equation*}
    \begin{aligned}
    \tilde{E}[u+\varepsilon \phi]\geq \tilde{E}[u],
    \end{aligned}
\end{equation*}
that is 
\begin{equation}\label{eq_1}
\begin{aligned}
    \varepsilon \int_\Omega \nabla u\cdot \nabla \phi+\frac{\varepsilon^2}{2}\int_\Omega |\nabla \phi|^2+\frac{\varepsilon}{k}\int_\Omega u\phi +\frac{\varepsilon^2}{2k}\int_\Omega \phi^2\geq -\int_\Omega (1-\frac{u_n}{k})[(u+\varepsilon \phi)^+-u].
    \end{aligned}
\end{equation}
Note that 
\begin{equation*}
\begin{aligned}
   \int_\Omega (1-\frac{u_n}{k})[(u+\varepsilon \phi)^+-u]=\varepsilon \int_{\{u+\varepsilon \phi\geq 0\}}(1-\frac{u_n}{k})\phi-\int_{\{u+\varepsilon\phi<0\}}(1-\frac{u_n}{k})u,
    \end{aligned}
\end{equation*}
ignoring the $O(\varepsilon^2)$ terms in \eqref{eq_1}, we have
\begin{equation}\label{eq_2}
\begin{aligned}
    \varepsilon \int_\Omega \nabla u\cdot \nabla \phi+\frac{\varepsilon}{k}\int_\Omega u\phi -\varepsilon\int_{\{u+\varepsilon \phi\geq 0\}}(1-\frac{u_n}{k})^-\phi+\int_{\{u+\varepsilon \phi< 0\}}(1-\frac{u_n}{k})^-u\\\geq -\varepsilon \int_{\{u+\varepsilon \phi\geq 0\}}(1-\frac{u_n}{k})^+\phi +\int_{\{u+\varepsilon \phi<0\}}(1-\frac{u_n}{k})^+ u.
    \end{aligned}
\end{equation}
In fact we have
 $$0\leq\int_{\{u+\varepsilon \phi<0\}}(1-\frac{u_n}{k})^{\pm}u<-\varepsilon \int_{\{u+\varepsilon \phi<0\}}(1-\frac{u_n}{k})^{\pm}\phi, $$
hence \eqref{eq_2} turns out to be
\begin{equation*}
\begin{aligned}
    \int_\Omega \nabla u\cdot \nabla \phi+\frac{1}{k}\int_\Omega u\phi -\int_{\{u+\varepsilon \phi\geq 0\}}(1-\frac{u_n}{k})^-\phi-\int_{\{u+\varepsilon \phi< 0\}}(1-\frac{u_n}{k})^-\phi\geq - \int_{\{u+\varepsilon \phi\geq 0\}}(1-\frac{u_n}{k})^+\phi.
    \end{aligned}
\end{equation*}
Moreover, we also recall that $u\geq 0$,  then in $L^1$ sense as $\varepsilon\to 0$,
\begin{equation*}
    \left\{
    \begin{aligned}
        &\chi_{\{u+\varepsilon \phi\geq 0\}}\to \chi_{A_\phi\cup\{u>0\}}\\
        &\chi_{\{u+\varepsilon \phi<0\}}\to \chi_{\{u=0\}\cap\{\phi<0\}},
    \end{aligned}
    \right.
\end{equation*}
where $A_\phi\coloneqq \{u=0\}\cap \{\phi\geq 0\}$. Clearly, $A_\phi$ and $\{u>0\}$ are disjoint.
This leads to 
\begin{equation*}
\begin{aligned}
    \int_\Omega \nabla u\cdot \nabla \phi+\frac{1}{k}\int_\Omega u\phi -\int_\Omega\chi_{A_\phi\cup\{u>0\}} (1-\frac{u_n}{k})^-\phi-\int_\Omega\chi_{\{u=0\}\cap\{\phi<0\}}(1-\frac{u_n}{k})^-\phi\\
    \geq - \int_\Omega\chi_{A_\phi\cup\{u>0\}}(1-\frac{u_n}{k})^+\phi,
    \end{aligned}
\end{equation*}
or equivalently,
\begin{equation}\label{eq_3}
\begin{aligned}
    \int_\Omega \nabla u\cdot \nabla \phi+\frac{1}{k}\int_\Omega u\phi +\int_\Omega\chi_{A_\phi\cup\{u>0\}} (1-\frac{u_n}{k})\phi-\int_\Omega\chi_{\{u=0\}\cap\{\phi<0\}}(1-\frac{u_n}{k})^-\phi\geq 0.
    \end{aligned}
\end{equation}

Define a distribution $$T(\phi)\coloneqq  \int_\Omega \nabla u\cdot \nabla \phi+\frac{1}{k}\int_\Omega u\phi +\int_\Omega\chi_{\{u>0\}} (1-\frac{u_n}{k})\phi, $$
then by \eqref{eq_3},
\begin{align*}
    T(\phi)\geq -\int_{A_\phi} (1-\frac{u_n}{k})\phi +\int_{\{u=0\}\cap\{\phi<0\}}(1-\frac{u_n}{k})^-\phi.
\end{align*}
Since $\phi$ is arbitrary, we may replace it with $-\phi$ and as a result,
\begin{equation}\label{eq_4}
\left\{
\begin{aligned}
    &T(\phi)\geq -\int_{A_\phi} (1-\frac{u_n}{k})\phi +\int_{\{u=0\}\cap\{\phi<0\}}(1-\frac{u_n}{k})^-\phi\\
    &T(\phi)\leq -\int_{\{u=0\}\cap\{\phi\leq0\}}(1-\frac{u_n}{k})\phi+\int_{\{u=0\}\cap\{\phi>0\}}(1-\frac{u_n}{k})^-\phi.
\end{aligned}\right.
\end{equation}
Therefore, $|T(\phi)|\leq C||\phi||_\infty$ for some positive constant $C$, thus by a density argument we derive that $T$ is a radon measure, i.e. there exists a density function $\rho(x)$ such that
$$T(\phi)=\int_{\Omega}\rho \phi \ dx.$$
However, by \eqref{eq_4}, we get $\rho=0$ a.e. in $\{u>0\}$; moreover, by definition of $T$ we get $\rho=0$ a.e. in $\{u=0\}$. This shows that $T(\phi)=0$, or
$$-\Delta u+\frac{1}{k} u+ \chi_{\{u>0\}}(1-\frac{u_n}{k})=0$$
in the weak sense. Equivalently,
$$\frac{u-u_n\cdot \chi_{\{u>0\}}}{k}=\Delta u-\chi_{\{u>0\}} .$$

\section{Proof of Theorem~\ref{th:epsilon}}
\label{s:eproof} 

As the discussion in Section~\ref{s:regular} above showed, $u=\lim_{\epsilon\to 0}u_\epsilon$ exists pointwisely by monotonicity. It remains to show $u$ is the solution to \eqref{eq_paravari}, that is
\begin{align*}
    \int_0^t\int_\Omega\partial_t u\cdot (v-u)+\int_0^t\int_\Omega \na u\cdot \na(v-u)\geq \int_0^t\int_\Omega u-v;\ \mbox{for all}\ v\in\mathcal{J},\ \mbox{a.e.}\ t\in(0,T). 
\end{align*} 
Intuitively, suppose that $f$ is a smooth approximation, then by maximum principle $|\na u_\epsilon|\leq \sup_{\Omega}|\na u_0|$ for any $x\in\Omega$ and $\epsilon>0$. 
Thus $|\na u|\leq \sup|\na u_0|$, therefore by Dini's Theorem, such convergence is uniform and as a result, $u\in \mathcal{J}$ because $u$ also satisfies the boundary condition and initial condition. 
Once we have such uniform boundedness of $\na u_{\epsilon}$, $\na u_\epsilon$ converges to $\na u$ weakly and as a result,
\begin{equation*}
   \lim_{\epsilon}\int_0^t\int_\Omega \na u_{\epsilon}\cdot \na(v-u)=\int_0^t\int_\Omega \na u\cdot\na(v-u)
\end{equation*}
and
\begin{equation*}
    \lim_{\epsilon}\int_0^t\int_\Omega -f_\epsilon(u_\epsilon)\cdot(v-u)=-\int_0^t\int_\Omega \chi_{\{u>0\}}\cdot(v-u)=-\int_0^t\int_\Omega v-u+\int_0^t\int_\Omega\chi_{\{u=0\}}\cdot v.
\end{equation*}
Indeed we have weak convergence of $\partial_t u_\epsilon$ thanks to the equation:
\begin{equation*}
    \lim_{\epsilon}\int_0^t\int_\Omega \partial_t u_\epsilon\cdot(v-u)=\lim_\epsilon \int_0^t\int_\Omega-f(u_\epsilon)\cdot (v-u)-\na u_\epsilon\cdot\na(v-u).
\end{equation*}
Since $u_\epsilon$ converges to $u$ pointwisely and strongly in $L^2((0,T);L^2(\Omega))$, then up to a subsequence
\begin{equation*}
    \lim_{\epsilon}\int_0^t\int_\Omega \partial_t u_\epsilon\cdot(v-u)=\int_0^t\int_\Omega\partial_tu\cdot(v-u).
\end{equation*}
Note that $v\geq 0$,
\begin{equation*}
    -\int_0^t\int_\Omega \chi_{\{u>0\}}\cdot(v-u)\geq- \int_0^t\int_\Omega v-u .
\end{equation*}
therefore
\begin{align*}
    \int_0^t\int_\Omega\partial_t u\cdot (v-u)+\int_0^t\int_\Omega \na u\cdot \na(v-u)\geq \int_0^t\int_\Omega u-v;\ \mbox{for all}\ v\in\mathcal{J},\ \mbox{a.e.}\ t\in(0,T). 
\end{align*}
Indeed, we only require the $H^1$ uniform boundedness of $u_\epsilon$. To see this without using smooth $f(u_\epsilon)$ we write down $u_\epsilon$ in the mild form:
\begin{equation*}
u_\epsilon(t)=e^{ t\Delta}u_0+\int_0^te^{(t-s)\Delta}(f_\epsilon(u_\epsilon))\ ds\ ,
\end{equation*}
where $e^{t\Delta}$ represents convolution with heat kernel. As a result, for any first order differential operator $D$ we have
\begin{align*}
    D u_\epsilon=De^{t\Delta}u_0+\int_0^tDe^{(t-s)\Delta}(f_\epsilon(u_\epsilon))\ ds
\end{align*}
and hence
\begin{equation*}
\begin{aligned}
\Vert Du_\epsilon\Vert_2\leq \Vert De^{ t\Delta}u_0\Vert_2+\int_0^t\Vert De^{(t-s)\Delta}f(u_\epsilon)\Vert_2\ ds.
\end{aligned}
\end{equation*}
Note that $e^{t\Delta}u_0$ solves the standard heat equation with initial data $u_0$, we have
\begin{equation*}
   \Vert De^{t\Delta}u_0\Vert_2=\Vert e^{t\Delta}Du_0\Vert_2\lesssim\Vert Du_0\Vert_{2}\lesssim1,
\end{equation*}
for any $t\in(0,T)$. On the other hand,
\begin{equation*}
\begin{aligned}
\Vert De^{(t-s)\Delta}f(u_\epsilon)\Vert_2\lesssim\Vert De^{(t-s)\Delta}f(u_\epsilon)\Vert_\infty=
|K*f(u_\epsilon)|\ , 
\end{aligned}
\end{equation*}
where $K$ is the kernel corresponding to $De^{(t-s)\Delta}$. Since $|f|\leq 1$,
\begin{equation*}
\begin{aligned}
|K*f(u_\epsilon)|&\leq\Vert K\Vert_2\cdot\Vert f(u_\epsilon)\Vert_2\\
&\lesssim\Vert K\Vert_2.
\end{aligned}
\end{equation*}
We see that from the Fourier side
\begin{equation*}
\begin{aligned}
\Vert K\Vert_2^2&\lesssim\sum_{k\in\Z^d}|k|^2e^{-2(t-s)|k|^2}\\
&=\sum_{|k|\geq 1}|k|^2e^{-2(t-s)|k|^2}\\
&\lesssim\int_1^\infty e^{-2(t-s)r^2}r^{1+d}\ dr\ .
\end{aligned}
\end{equation*}
For the 1D case, first we observe that
\begin{align*}
    \int_1^\infty e^{-2(t-s)r^2}r^2\ dr=&\frac{\sqrt{2\pi}[1-\erf(\sqrt{2(t-s)})]+4\sqrt{t-s}e^{-2(t-s)}}{16(t-s)^{3/2}}\\
    \lesssim&\frac{1-\erf(\sqrt{2(t-s)})}{(t-s)^{3/2}}+\frac{e^{-2(t-s)}}{t-s},
\end{align*} 
where $\erf(x)=\frac{2}{\sqrt{\pi}}\int_0^xe^{-t^2}\ dt$, the Gauss error function. Therefore,
\begin{equation*}
    \Vert De^{\gamma\Delta}f(u_\epsilon)\Vert_2\lesssim\frac{\left(1-\erf(\sqrt{2(t-s)})\right)^{1/2}}{(t-s)^{3/4}}+\frac{e^{-(t-s)}}{(t-s)^{1/2}}.
\end{equation*}
Now we would assume $t\geq 1$, as the other case $t<1$ is easier. Let $\gamma=t-s$, we split the following integral into 2 parts:
\begin{align*}
\int_0^t\Vert De^{\gamma\Delta}f(u_\epsilon)\Vert_2\ d\gamma=\int_0^1\Vert De^{\gamma\Delta}f(u_\epsilon)\Vert_2\ d\gamma+\int_1^t\Vert De^{\gamma\Delta}f(u_\epsilon)\Vert_2\ d\gamma.
\end{align*}

\begin{description}
      \item[(i) $\gamma>1$:]
 Then we have 
 \begin{equation*}
     \frac{\left(1-\erf(\sqrt{2\gamma})\right)^{1/2}}{\gamma^{3/4}}\lesssim \frac{e^{-\gamma}}{\gamma^{5/4}},
 \end{equation*}
 thus
\begin{equation*}
\begin{aligned}
\int_1^t\Vert De^{\gamma\Delta}f(u_\epsilon)\Vert_2\ d\gamma&\lesssim\int_1^t \frac{e^{- \gamma}}{\gamma^{3/4}}+\frac{e^{-\gamma}}{\gamma^{1/2}}\ d\gamma\\&\lesssim
\int_1^t\frac{e^{- \gamma}}{\gamma^{1/2}}\ d\gamma\\
&\lesssim \int_1^\infty \frac{e^{- \gamma}}{\gamma^{1/2}}\ d\gamma\\
&\lesssim 1\ .
\end{aligned}
\end{equation*}
\item[(ii) $\gamma\leq1$:] We use another estimate for $\Vert K*f(u_\epsilon)\Vert _2$. We compute from the Fourier side:
\begin{equation*}
\begin{aligned}
\Vert K*f(u_\epsilon)\Vert _2^2&=\sum_{|k|\geq 1}|k|^2e^{-2\gamma|k|^2}|\widehat{f(u_\epsilon)}(k)|^2\\
&\leq \max_{|k|\geq 1}\left\{|k|^2e^{-2\gamma|k|^2}\right\}\cdot \sum_{|k|\geq 1}|\widehat{f(u_\epsilon)}(k)|^2\\
&\lesssim \max_{|k|\geq 1}\left\{|k|^2e^{-2\ga|k|^2}\right\}\cdot\Vert f(u_\epsilon)\Vert _2^2\\
&\lesssim \max_{|k|\geq 1}\left\{|k|^2e^{-2\ga|k|^2}\right\}.
\end{aligned}
\end{equation*}
Define $g(x)=x^2e^{-2\gamma x^2}$, where $x\geq 0$. Then,

\begin{equation*}
g'(x)=xe^{-2 \gamma x^2}\left(1-2\gamma x^2\right)\ ,
\end{equation*}
this shows the maximum achieves at $x=\frac{1}{\sqrt{2\gamma}}$ and hence 
\begin{equation*}
g(x)\leq g(\frac{1}{\sqrt{2\gamma}})\leq\frac{1}{\gamma}\,
\end{equation*}
thus 
\begin{equation*}
\Vert De^{\gamma\Delta}f(u_\epsilon)\Vert _2\lesssim\frac{1}{\sqrt{\gamma}}\ ,
\end{equation*}
As a result, 
\begin{equation*}
\begin{aligned}
\int_0^1\Vert De^{\gamma\Delta}f(u_\epsilon)\Vert _2\ d\gamma\lesssim \int_0^1\frac{1}{\sqrt{\gamma}}\ d\gamma\cdot\Vert f(u_\epsilon)\Vert _2\lesssim 1\ .
\end{aligned}
\end{equation*}
\end{description}
Similar arguments can be applied to the 2D and 3D cases. In what follows,
\begin{equation*}
    \Vert Du_\epsilon\Vert_2\lesssim1,
\end{equation*}
for any $t\in(0,T)$ and the bound is independent of $\epsilon$. 

\section{Another Proof of the Regularization Result}
\label{s:AugL} 
We recall the variational inequality setting \eqref{eq_paravari}, that is to solve $u\in H^1_+(\Omega)$
 \begin{equation*}
     \int_0^t\langle \partial_t u-\Delta u+1, v-u\rangle \geq 0, \ \forall\ v\in \mathcal{J}. 
 \end{equation*}
 As in \cite{AugLag_paravar}, it then has an equivalent formulation, that is to solve $u(t)$ and $\la^*(t)$:
 \begin{equation}\label{eq:ALM}\left\{
     \begin{aligned}
      &\partial_t u-\Delta u+1=-\la^*(t)\geq 0\\
      &u\geq 0,\ \langle u(t),\la^*(t)\rangle=0,\ \forall\ t>0.
     \end{aligned}\right.
 \end{equation}
 To approach this, we introduce a regularized approximation family: we aim to find $u_c$ for any $c>0$ such that the following holds weakly:
 \begin{align*}
     \partial_t u_c -\Delta u^c+1+\min\left(0,-1+cu^c\right)=0.
 \end{align*}
 By defining $\la^c=\min\left(0,-1+cu^c\right)$, we can rewrite the above scheme as 
 \begin{align*}
      \partial_t u^c -\Delta u^c+1+\la^c=0.
 \end{align*}
 It is typical to write the regularization term in this way in some literature, but the approach is the same as the regularization in Section~\ref{s:regular} with $c=1/\epsilon$. 
 We then discretize it in time: for any $\phi\in H^1$, the following holds
 \begin{equation}\label{eq:ALMBE}
     \left\langle\frac{u_{n+1}^c-u_n^c}{k},\phi\right\rangle+\left\langle\nabla u_{n+1}^c,\nabla \phi\right\rangle+\langle1,\phi\rangle+\left\langle\min(0, -1+cu_{n+1}^c),\phi\right\rangle=0,
 \end{equation}
where $u^c_0$ is chosen to be $u_0$. We write $u_n$ instead of $u_n^c$ for simplicity. Note that the operator $A(u)\coloneqq \frac{u}{k}-\Delta u+\min(0,-1+cu)$ is coercive and monotone. As a result, there exists a unique solution $u_{n+1}\in H^1$ for sufficiently small $k>0$ independent of $c>0$. To show $u_{n+1}\in H^1_+(\Omega)$, we prove by induction. Assuming $u_n\in H^1_+(\Omega)$, we test the \eqref{eq:ALMBE} with $(u_{n+1})^-$. Therefore we derive that 
 \begin{align*}
     &\frac1k\langle u_{n+1},(u_{n+1})^-\rangle+\left\langle\nabla u_{n+1},\nabla (u_{n+1})^-\right\rangle+\langle1,(u_{n+1})^-\rangle+\left\langle\min(0, -1+cu_{n+1}),(u_{n+1})^-\right\rangle\\=&\frac1k\langle u_{n},(u_{n+1})^-\rangle\leq 0.
 \end{align*}
 We observe that $\left\langle\nabla u_{n+1},\nabla (u_{n+1})^-\right\rangle=\left\langle\nabla (u_{n+1})^-,\nabla (u_{n+1})^-\right\rangle\geq0$. Moreover, $\langle1,(u_{n+1})^-\rangle+\langle\min(0, -1+cu_{n+1}),(u_{n+1})^-\rangle=c\langle (u_{n+1})^-,(u_{n+1})^-\rangle\geq0$. We thus obtain that
 $\langle u_{n+1}, (u_{n+1})^-\rangle\leq 0$ and hence $u_{n+1}\in H^1_+(\Omega)$. We then define
 $$u^c_{M}(x,t)=u_n+\frac{t-nk}{k}(u_{n+1}-u_n),\ \mbox{for}\ t\in[nk,(n+1)k),$$
where $M=T/k$. By the same argument in Lemma~\ref{lem_RoMet}, we have $u^c_M$ converges to function $u^c$ in $L^2(0,T;H^1)$ as $M\to\infty$ up to a subsequence. In fact, it is easy to see that $u_c$ is the solution to \eqref{eq:ALMBE}. On the other hand, we show that $u^c$ converges to $u^*$ as $c\to \infty$.
\begin{theorem}[Monotonicity]
Let $u_{n+1}^c$ and $u^c$ be defined as above. If $0<c\leq b $, then $u_{n+1}^c\geq u_{n+1}^b$ for all $n=0,1,2,\cdots$. Therefore $u^c(t)\geq u^b(t)$ as a direct application. 
\end{theorem}
\begin{proof}
The proof is given by induction. Suppose $u_n^c\geq u_n^b$ and for each $n$ define $\la_n^c$ by 
$$\la_{n+1}^c=\min(0,-1+cu_{n+1}^c).$$
Then the proof is similar to the one showing $u_{n+1}^c\geq0$, we have that
\begin{align*}
     &\frac1k\langle u_{n+1}^c-u_{n+1}^b,(u_{n+1}^c-u_{n+1}^b)^-\rangle+\left\langle\nabla( u_{n+1}^c-u_{n+1}^b),\nabla (u_{n+1}^c-u_{n+1}^b)^-\right\rangle
     \\&+\left\langle \la^c_{n+1}-\la^b_{n+1},(u_{n+1}^c-u_{n+1}^b)^-\right\rangle=\frac1k\langle u^{n}_c-u_n^b,(u_{n+1}^c-u_{n+1}^b)^-\rangle\leq 0.
 \end{align*}
Note that $cu^c_{n+1}-bu^b_{n+1}\leq cu^c_{n+1}-cu^b_{n+1}$ for $c\leq b$ and hence
$\langle\la^c_{n+1}-\la^b_{n+1},(u_{n+1}^c-u_{n+1}^b)^-\rangle\geq0$. We thus obtain that $u^c_{n+1}\geq u_b$.

\end{proof}
As a corollary of the monotonicity, we obtain the existence of $u(t)$ and it solves \eqref{eq:ALM}. Uniqueness can be proved similarly as in \cite{AugLag_paravar}.

\bibliographystyle{abbrv}
\bibliography{refOD}

\end{document}